\documentclass[10pt, a4paper]{article}
\usepackage{amsmath, amsfonts}
\usepackage[usenames]{color}
\usepackage[latin1]{inputenc}
\usepackage{mathrsfs}
\usepackage{hyperref} 
\hypersetup{bookmarks = true, 
                  colorlinks = true, 
                  linkcolor = blue, 
                  citecolor = blue, 
                  }
\newcommand{\xdownarrow}[1]{%
  {\left\downarrow\vbox to #1{}\right.\kern-\nulldelimiterspace}
}

\newcommand{\xuparrow}[1]{%
  {\left\uparrow\vbox to #1{}\right.\kern-\nulldelimiterspace}
}

\usepackage{graphicx}
\usepackage{amssymb}
\usepackage{amsmath}

\newtheorem{theorem}{Theorem}[section]

\newtheorem{conjecture}[theorem]{Conjecture}

\newtheorem{definition}[theorem]{Definition}
\newtheorem{example}[theorem]{Example}

\newtheorem{lemma}[theorem]{Lemma}

\newtheorem{proposition}[theorem]{Proposition}
\newtheorem{remark}[theorem]{Remark}

\newenvironment{proof}[1][Proof]{\noindent\textbf{#1.} }{\ \rule{0.5em}{0.5em}}

\title{\textbf{Equimultiplicity of families of map germs from $\mathbb{C}^2$ to $\mathbb{C}^3$}}
\author{ \ \ \ \\{O.N. Silva\footnote{O.N. Silva: Instituto de Matemáticas, Universidad Nacional Autónoma de México (UNAM), Cuernavaca, México. \hspace{5cm} e-mail: otoniel@im.unam.mx}}}
\date{}

\begin{document}

\maketitle

\begin{abstract}
In 1971, Zariski proposed some questions in Theory of Singularities. One of such problems is the so-called, nowadays, \textit{``Zariski's multiplicity conjecture''}. In this work, we consider the version of this conjecture for families. We answer positively Zariski's multiplicity conjecture for a special class of non-isolated singularities.
\end{abstract}

\section{Introduction}

$ \ \ \ \ $ In 1971, Zariski proposed some questions in Theory of Singularities \cite{zariski71}. One of such problems is the so-called, nowadays, \textit{``Zariski's multiplicity conjecture''}. In this work, we consider the version of this conjecture for families. More precisely, let $g:(\mathbb{C}^n,0)\rightarrow (\mathbb{C},0)$ be a reduced germ of holomorphic function, $V(g)=g^{-1}(0)$ the corresponding germ of hypersurface in $\mathbb{C}^n$. Let

\[  \begin{array}{ccc}

G:(\mathbb{C}^n \times \mathbb{C}, \lbrace 0 \rbrace \times \mathbb{C}) & \rightarrow & (\mathbb{C},0)\\

(z,t) & \mapsto & G(z,t)=G_t(z)

 \end{array}\] 

\noindent be a deformation of $g$, that is, $G$ is a germ of holomorphic function such that $G_0=g$ and $G_t(0)=0$. The multiplicity of $V(G_t)=G_t^{-1}(0)$ at $0$ is the number of points of the intersection of $V(G_t)$ with a generic line passing close to the origin but not through the origin. Denote by $m(V(G_t))$ the multiplicity of $V(G_t)$ at $0$. One says that the family $V(G_t)$ is equimultiple if, for all $t$ near $0$, $m(V(G_t))=m(V(G_0))$. One says that the family $V(G_t)$ is topologically trivial if, for all $t$ near $0$, there is a germ of homeomorphism $h_t:(\mathbb{C}^n,0)\rightarrow(\mathbb{C}^n,0)$ such that $h(V(G_t))=V(G_0)$. Now, we can state:

\begin{conjecture}\label{conjecture}
\textit{(Zariski's multiplicity conjecture for families) If $G_t$ is topologically trivial, then is it equimultiple?}
\end{conjecture}

Almost fifty years later, Conjecture \ref{conjecture} is, in general, still unsettled (even for isolated hypersurfaces singularities). The answer is, nevertheless, known to be yes in several special cases. We cite for example, Zariski in \cite{zariski62}, Greuel in \cite{greuel2}, O'Shea in \cite{oshea}, Trotman in \cite{trotman3} and \cite{trotman5}. More recently, we can cite \cite{abd}, \cite{eyral3}, \cite{mendris}, \cite{plenat}, \cite{tomazella} and \cite{trotman4}. We also note that some important contributions for the non-isolated case of Conjecture \ref{conjecture} was were recently given by Eyral and Ruas in \cite{eyral5} and \cite{eyral4}. In order to know more about this conjecture and a more complete list, see the survey \cite{eyral2}.

In this work, we provide a positive answer to Zariski's multiplicity conjecture for some families of singular surfaces in $\mathbb{C}^3$ with non-isolated singularities under the condition that the family has a smooth normalization. For a family of surfaces $V(G_t)$ in $\mathbb{C}^3$ whose normalization is smooth, we can associate a family of parametrizations $f_t:(\mathbb{C}^2,0)\rightarrow (\mathbb{C}^3,0)$ whose images are $V(G_t)$. In this situation, it follows from \cite[Theorem 38]{ref2} that the $\mathcal{A}$-topological equisingularity  of $f_t$ implies that the Milnor numbers of the sets $f_t^{-1}(\Sigma_t)$ remain constant, where $\Sigma_t$ is the singular set of $V(G_t)$. The converse also holds with some additional hypothesis on $G_t$ (see \cite[Corollary 39]{ref2}).

A particular class of parametrized singular surfaces consists on surfaces which are the image of an $\mathcal{A}$-finitely determined map germ $f:(\mathbb{C}^2,0)\rightarrow (\mathbb{C}^3,0)$. Let $F:(\mathbb{C}^2 \times \mathbb{C},0)\rightarrow (\mathbb{C}^3 \times \mathbb{C},0)$, $F(x,y,t)=(f_t(x,y),t)$ be a $1$-parameter unfolding of a finitely determined map germ $f:(\mathbb{C}^2,0)\rightarrow (\mathbb{C}^3,0)$. In this case, the family of surfaces $F(\mathbb{C}^2 \times \mathbb{C})$ is topologically trivial if and only if $\mu(D(f_t))$ is constant in a neighbourhood of $0$, where $D(f_t)$ is the double point curve of $f_t$ (see Section \ref{sec2.1}) and $\mu$ is the Milnor number of these sets (see \cite[Corollary 40]{ref2} and also \cite[Theorem 6.2]{ref1}). In this context, we can ask the following question:

\begin{flushleft}
\textit{Question: What kind of conditions do we need on $f$ that imply equimultiplicity of a topologically trivial family $F(\mathbb{C}^2 \times \mathbb{C})$ ?}
\end{flushleft}

In this sense, we prove the following result, which is our first main result:

\begin{theorem}\label{mainresult1intro}
Let $f:(\mathbb{C}^2,0)\rightarrow (\mathbb{C}^3,0)$ be a finitely determined map germ of corank $1$. Write $f$ in the form $f(x,y)=(x,p(x,y),q(x,y))$ and suppose that $p$ and $q$ are quasihomogeneous functions such that the weights of the variables are $w(x)=1$ and $w(y)=b \geq 2$. Let $F=(f_t,t)$ be a $1$-parameter unfolding of $f$. If $F$ is topologically trivial, then $F$ is Whitney equisingular. In particular, the family $F(\mathbb{C}^2 \times \mathbb{C})$ is equimultiple.
\end{theorem}

We recall that if a map germ from $\mathbb{C}^2$ to $\mathbb{C}^3$ has corank $1$, then we can always find coordinates in which $f$ takes the form $f(x,y)=(x,p(x,y),q(x,y))$ (see for instance \cite[Lemma 4.1]{mond7}). 

Let $f:(\mathbb{C}^2,0)\rightarrow (\mathbb{C}^3,0)$ be a finitely determined map germ of corank $1$ and $F(x,y,t)=(f_t(x,y),t)$ a topologically trivial $1$-parameter unfolding of $f$. By combining \cite[Corollary 40]{ref2} and \cite[Prop. 8.6 and Cor. 8.9]{gaffney} one can conclude that $F$ is Whitney equisingular (see Definition \ref{defequisingularity}) if and only if the multiplicity $m(f_t(D(f_t)))$ is constant, where $f_t(D(f_t))$ denotes the image of the double point curve of $f_t$ (see Section \ref{sec2.1}).

In this way, our technique is to study the irreducible components (branches) of the double point curves $D(f)$ (in the source) and $f(D(f))$ (in the target). In \cite[Def. 2.4]{otoniel1}, Ruas and the author introduced the notion of identification and fold components of $D(f)$. In the case where $f$ is homogeneous and has corank $1$, a complete description of these two kinds of components of $D(f)$ was given in \cite[Prop. 6.2]{otoniel1}. In Section \ref{quasihomog}, we extend this result for a particular class of quasihomogeneous maps (Proposition \ref{propdasformulas}).

In Section \ref{sec:3}, with the description of the identification and fold components of $D(f)$ given in Proposition \ref{propdasformulas}, we show that in the conditions of Theorem \ref{mainresult1intro}, each irreducible component of $f(D(f))$ is smooth. Hence, $F$ is Whitney equisingular. In particular, $F$ is equimultiple (Theorem \ref{mainresult1}).

Apart from the hypothesis of Theorem \ref{mainresult1intro}, the situation may be a little more complicated, since $F$ may not be Whitney equisingular (see \cite[Example 5.5]{otoniel1}). At this point, we note that the class of topologically trivial unfoldings of a homogeneous finitely determined map germ $f$ (i.e. when $b=1$ in Th. \ref{mainresult1intro}) was studied recently by Ruas and the author in \cite{otoniel1}, and by the author in his PhD. thesis \cite{otoniel2}, where the first known counter-example to a conjecture by Ruas \cite[p. 120]{ref17} on the equivalence between Whitney equisingularity and topological triviality was obtained. Also, we remark that if $b=1$ in Theorem \ref{mainresult1intro} and the greatest common divisor of the degrees of $p$ and $q$ is different to $2$, then $F$ is also Whitney equisingular (see \cite[Theorem 7.2]{otoniel1}). In this way, Theorem \ref{mainresult1intro} can be sees as an extension of (\cite[Theorem 7.2]{otoniel1}) for another class of maps where Ruas's conjecture is true.

In Section \ref{sec5}, we relate the equimultiplicity of a family $f_t$ in terms of the Cohen-Macaulay property of a certain local ring (Lemma \ref{lemaaux}). In sequence, we finished this work showing the following theorem, which is considered our second main result.

\begin{theorem}\label{mainresult2intro} Let $f:(\mathbb{C}^2,0)\rightarrow (\mathbb{C}^3,0)$ be a quasihomogeneous and finitely determined map germ of corank $1$. Let $F(x,y,t)=(f_t(x,y),t)$ be a topologically trivial $1$-parameter unfolding of $f$. If $F$ is an unfolding of non-decreasing weights, then the family $F(\mathbb{C}^2 \times \mathbb{C})$ is equimultiple.
\end{theorem}

\section{Preliminaries}

$ \ \ \ \ $ Throughout the paper, given a finite map $f:\mathbb{C}^2\rightarrow \mathbb{C}^3$, $(x,y)$ and $(X,Y,Z)$ are used to denote systems of coordinates in $\mathbb{C}^2$ (source) and $\mathbb{C}^3$ (target), respectively. Given an $1$-parameter unfolding $F:\mathbb{C}^2 \times \mathbb{C} \rightarrow \mathbb{C}^3 \times \mathbb{C}$ of $f$, we call the last coordinate (i.e; the parameter space) in $\mathbb{C}^2 \times \mathbb{C}$ and $\mathbb{C}^3 \times \mathbb{C}$ by $t$. Also, $\mathbb{C} \lbrace x_1,\cdots,x_n \rbrace \simeq \mathcal{O}_n$ denotes the local ring of convergent power series in $n$ variables.

\subsection{Double point spaces}\label{sec2.1}

$ \ \ \ \ $ In this section, we describe the sets of double points of a finite map from $\mathbb{C}^2$ to $\mathbb{C}^3$ following the description given in \cite[Section 2]{ref9} (and also in \cite[Section 1]{ref7} and \cite[Section 3]{ref12}). Let $U \subset \mathbb{C}^2$ and $V \subset \mathbb{C}^3$ be open sets. Throughout we assume that a map $f:U\rightarrow V$ is finite, that is, holomorphic, closed and finite-to-one, unless otherwise stated.

Following Mond \cite[Section 3]{ref12}, we define spaces related to the double points of a given finite mapping $f:U \rightarrow V$, by firstly considering the sheaf $\mathcal{I}_{2}$ and $\mathcal{I}_{3}$  defining the diagonal of $\mathbb{C}^2 \times \mathbb{C}^2$ and $\mathbb{C}^3 \times \mathbb{C}^3$, respectively. That is, locally 

\begin{center}
 $\mathcal{I}_{2}=\langle x-x^{'},y-y^{'} \rangle$ $ \ \ \ $ and $ \ \ \ $ $\mathcal{I}_{3}=\langle X-X^{'},Y-Y^{'}, Z-Z^{'} \rangle$.
 \end{center} 

Since the pull-back $(f \times f)^{\ast}\mathcal{I}_{3}$ is contained in $\mathcal{I}_{2}$ and $U$ is small enough, then there exist sections $\alpha_{ij}\in \mathcal{O}_{\mathbb{C}^{2} \times \mathbb{C}^2}(U \times U)$ well defined in all $U \times U$, such that

\begin{center}
$f_{i}(x,y)-f_{i}(x^{'},y^{'})= \alpha_{i1}(x,y,x^{'},y^{'})(x-x^{'})+ \alpha_{i2}(x,y,x^{'},y^{'})(y-y^{'}), \ for \ i=1,2,3.$
\end{center}

If $f(x,y)=f(x^{'},y^{'})$ and $(x,y) \neq (x^{'},y^{'})$, then every $2 \times 2$ minor of the matrix $\alpha=(\alpha_{ij})$ must vanish at $(x,y,x^{'},y^{'})$. Denote by $\mathcal{R}(\alpha)$ the ideal in $\mathcal{O}_{\mathbb{C}^{4}}$ generated by the $2\times 2$ minors of $\alpha$. Then we have the following definition.

\begin{definition} The \textit{lifting of the double point space of $f$} is the complex space

\[
D^{2}(f)=V((f\times f)^{\ast}\mathcal{I}_{3}+\mathcal{R}(\alpha)).
\]

\end{definition}

Although the ideal $\mathcal{R}(\alpha)$ depends on the choice of the coordinate functions of $f$, in \cite[Section 3]{ref12} it is proved that the ideal $(f\times f)^{\ast}\mathcal{I}_{3}+\mathcal{R}(\alpha)$ does not, and so $D^{2}(f)$ is well defined. It is not hard to see that the points in the underlying set of $D^{2}(f)$ are exactly the ones in $U \times U$ of type $(x,y,x^{'},y^{'})$ with $(x,y) \neq (x^{'},y^{'})$, $f(x,y)=f(x^{'},y^{'})$ and of type $(x,y,x,y)$ such that $(x,y)$ is a singular point of $f$.

Let $f:(\mathbb{C}^2,0)\rightarrow(\mathbb{C}^{3},0)$ be a finite map germ and denote by $I_{3}$ and $R(\alpha)$ the stalks at $0$ of $\mathcal{I}_{3}$ and $\mathcal{R}(\alpha)$. By taking a representative of $f$, we define the \textit{lifting of the double point space of the map germ $f$} as the complex space germ $D^{2}(f)=V((f \times f)^{\ast}I_{3}+R(\alpha))$.

Once the lifting $D^2(f) \subset \mathbb{C}^2 \times \mathbb{C}^2$ is defined, we now consider its image $D(f)$ on $\mathbb{C}^2$ by the projection $\pi:\mathbb{C}^2 \times \mathbb{C}^2 \rightarrow \mathbb{C}^2$ onto the first factor. The most appropriate structure for $D(f)$ is the one given by the Fitting ideals, because it relates in a simple way the properties of the spaces $D^2(f)$ and $D(f)$. 

Another important space to study the topology of $f(\mathbb{C}^{2})$ is the double point curve in the target, that is, the image of $D(f)$ by $f$, denoted by $f(D(f))$, which will also be consider with the structure given by Fitting ideals. 

More precisely, given a finite morphism of complex spaces $f:X\rightarrow Y$ the push-forward $f_{\ast}\mathcal{O}_{X}$ is a coherent sheaf of $\mathcal{O}_{Y}-$modules (see \cite[Chapter 1]{grauert}) and to it we can (as in \rm\cite[Section 1]{ref13}) associate the Fitting ideal sheaves $\mathcal{F}_{k}(f_{\ast}\mathcal{O}_{X})$. Notice that the support of $\mathcal{F}_{0}(f_{\ast}\mathcal{O}_{X})$ is just the image $f(X)$. Analogously, if $f:(X,x)\rightarrow(Y,y)$ is a finite map germ then we denote also by $ \mathcal{F}_{k}(f_{\ast}\mathcal{O}_{X})$ the \textit{k}th Fitting ideal of $\mathcal{O}_{X,x}$ as $\mathcal{O}_{Y,y}-$module. In this way, we have the following definition:

\begin{definition} Let $f:U \rightarrow V$ be a finite mapping.\\

\noindent (a) Let ${\pi}|_{D^2(f)}:D^2(f) \subset U \times U \rightarrow U$ be the restriction to $D^2(f)$ of the projection $\pi$. The \textit{double point space} is the complex space

\begin{center}
$D(f)=V(\mathcal{F}_{0}({\pi}_{\ast}\mathcal{O}_{D^2(f)}))$.
\end{center}

\noindent Set theoretically we have the equality $D(f)=\pi(D^{2}(f))$.\\

\noindent (b) The \textit{double point space in the target} is the complex space $f(D(f))=V(\mathcal{F}_{1}(f_{\ast}\mathcal{O}_2))$. Notice that the underlying set of $f(D(f))$ is the image of $D(f)$ by $f$.\\ 

\noindent (c) Given a finite map germ $f:(\mathbb{C}^{2},0)\rightarrow (\mathbb{C}^3,0)$, \textit{the germ of the double point space} is the germ of complex space $D(f)=V(F_{0}(\pi_{\ast}\mathcal{O}_{D^2(f)}))$. \textit{The germ of the double point space in the target} is the germ of the complex space $f(D(f))=V(F_{1}(f_{\ast}\mathcal{O}_2))$.

\end{definition}

\begin{remark} If $f:U \subset \mathbb{C}^2 \rightarrow V \subset \mathbb{C}^3 $ is finite and generically $1$-to-$1$, then $D^2(f)$ is Cohen-Macaulay and has dimension $1$ (see \rm\cite[\textit{Prop. 2.1}]{ref9}\textit{). Hence, $D^2(f)$, $D(f)$ and $f(D(f))$ are curves. In this case, without any confusion, we also call these complex spaces by the ``lifting of the double point curve'', the ``double point curve'' and the ``image of the double point curve'', respectively.}
\end{remark}

\subsection{Finite Determinacy}

\begin{definition}(a) Two map germs $f,g:(\mathbb{C}^2,0)\rightarrow (\mathbb{C}^3,0)$ are $\mathcal{A}$-equivalent, denoted by $g\sim_{\mathcal{A}}f$, if there exist map germs of diffeomorphisms $\eta:(\mathbb{C}^2,0)\rightarrow (\mathbb{C}^2,0)$ and $\xi:(\mathbb{C}^3,0)\rightarrow (\mathbb{C}^3,0)$, such that $g=\xi \circ f \circ \eta$.\\

\noindent (b) $f:(\mathbb{C}^2,0) \rightarrow (\mathbb{C}^3,0)$ is finitely determined ($\mathcal{A}$-finitely determined) if there exists a positive integer $k$ such that for any $g$ with $k$-jets satisfying $j^kg(0)=j^kf(0)$ we have $g \sim_{\mathcal{A}}f$. 

\end{definition}

\begin{remark}\label{remarktriplepoints} Consider a finite map germ $f:(\mathbb{C}^2,0)\rightarrow (\mathbb{C}^3,0)$. By Mather-Gaffney criterion (\rm\cite[\textit{Theorem 2.1}]{ref20}\textit{), $f$ is finitely determined if and only if there is a finite representative $f:U \rightarrow V$, where $U\subset \mathbb{C}^2$, $V \subset \mathbb{C}^3$ are open neighbourhoods of the origin, such that $f^{-1}(0)=\lbrace 0 \rbrace$ and the restriction $f:U \setminus \lbrace 0 \rbrace \rightarrow V \setminus \lbrace 0 \rbrace$ is stable.} 

\textit{This means that the only singularities of $f$ on $U \setminus \lbrace 0 \rbrace$ are cross-caps (or Whitney umbrellas), transverse double and triple points. By shrinking $U$ if necessary, we can assume that there are no cross-caps nor triple points in $U$. Then, since we are in the nice dimensions of Mather (}\rm\cite[\textit{p. 208}]{mather}\textit{), we can take a stabilization of $f$, $F:U \times D \rightarrow \mathbb{C}^4$, $F(z,s)=(f_{s}(z),s)$ where $D$ is a neighbourhood of $0$ in $\mathbb{C}$. It is known that the number $T(f):= \sharp $ of triple points of $f_s$, for $s\neq 0$ and the number $C(f):= \sharp $ of cross-caps of $f_s$, are analytic invariants of $f$ and can be computed as (}\rm \cite[\textit{Th. 4.3}]{ref13} \textit{and} \rm\cite[Section 2]{ref12}\textit{):}

\begin{center}
 $  T(f)=dim_{\mathbb{C}}\dfrac{\mathcal{O}_3}{F_2(f_{\ast}\mathcal{O}_2)},   \ \ \ \ \  C(f)=dim_{\mathbb{C}}\dfrac{\mathcal{O}_2}{J(f)}$
\end{center}

\noindent where $J(f)$ is the ramification ideal of $f$. These numbers are finite, provided that $f$ is finitely determined.

\end{remark}

We note that the set $D(f)$ plays a fundamental role in the study of the finite determinacy. In \cite[Theorem 2.14]{ref7}, Marar and Mond presented necessary and sufficient conditions for a map germ $f:(\mathbb{C}^n,0)\rightarrow (\mathbb{C}^p,0)$ with corank $1$ to be finitely determined in terms of the dimensions of $D^2(f)$ and other multiple points spaces. In \cite{ref9}, Marar, Nu\~{n}o-Ballesteros and Pe\~{n}afort-Sanchis extended in some way this criterion of finite determinacy to the corank $2$ case. More precisely, they proved the following result:

\begin{theorem}\rm(\cite[Corollary 3.5]{ref9})\label{criterio} \textit{
Let $f:(\mathbb{C}^2,0)\rightarrow(\mathbb{C}^{3},0)$ be a finite and generically $1$ $-$ $1$ map germ. Then $f$ is finitely determined if and only if $\mu(D(f))$ is finite (equivalently, $D(f)$ is a reduced curve).}
\end{theorem}

\section{Identification and fold components of $D(f)$}\label{quasihomog}

$ \ \ \ \ $ When $f:(\mathbb{C}^2,0)\rightarrow (\mathbb{C}^3,0)$ is finitely determined, the restriction of (a representative) $f$ to $D(f)$ is finite. In this case, $f_{|D(f)}$ is generically $2$-to-$1$ (i.e; $2$-to-$1$ except at $0$). On the other hand, the restriction of $f$ to an irreducible component $D(f)^i$ of $D(f)$ can be generically $1$-to-$1$ or $2$-to-$1$. This motivates us to give the following definition which is from \cite[Definition 4.1]{otoniel2} (see also \cite[Definition 2.4]{otoniel1}).

\begin{definition}\label{typesofcomp} Let $f:(\mathbb{C}^2,0)\rightarrow (\mathbb{C}^3,0)$ be a finitely determined map germ. Let $f:U\rightarrow V$ be a representative, where $U$ and $V$ are neighbourhoods of $0$ in $\mathbb{C}^2$ and $\mathbb{C}^3$, respectively, and consider an irreducible component $D(f)^j$ of $D(f)$.\\

\noindent (a) If the restriction ${f_|}_{D(f)^j}:D(f)^j\rightarrow V$ is generically $1$ $-$ $1$, we say that $D(f)^j$ is an \textit{identification component of} $D(f)$.\\ 

In this case, there exists an irreducible component $D(f)^i$ of $D(f)$, with $i \neq j$, such that $f(D(f)^j)=f(D(f)^i)$. We say that $D(f)^i$ is the \textit{associated identification component to} $D(f)^j$ or that the pair $(D(f)^j, D(f)^i)$ is a \textit{pair of identification components of} $D(f)$.  \\

\noindent (b) If the restriction ${f_|}_{D(f)^j}:D(f)^j\rightarrow V$ is generically $2$ $-$ $1$, we say that $D(f)^j$ is a \textit{fold component of} $D(f)$.
\end{definition}

The following example illustrates the two types of irreducible components of $D(f)$ presented in Definition \ref{typesofcomp}.

\begin{example}\label{example}
\textit{Let $f(x,y)=(x,y^2,xy^3-x^5y)$ be the singularity $C_5$ of Mond's list} \rm(\cite[p. 378]{ref12}). \textit{In this case, $D(f)=V(xy^2-x^5)$. Then $D(f)$ has three irreducible components given by}

\begin{center}
$D(f)^1=V(x^2-y), \ \ \ $ $D(f)^2=V(x^2+y) \ \ $ and $ \ \ D(f)^3=V(x)$. 
\end{center}

\textit{Notice that $D(f)^3$ is a fold component and $(D(f)^1$, $D(f)^2)$ is a pair of identification components. Also, we have that $f(D(f)^3)=V(X,Z)$ and $f(D(f)^1)=f(D(f)^2)=V(Y-X^4,Z)$ (see Figure \rm\ref{figura38}).}\\

\begin{figure}[h]
\centering
\includegraphics[scale=0.27]{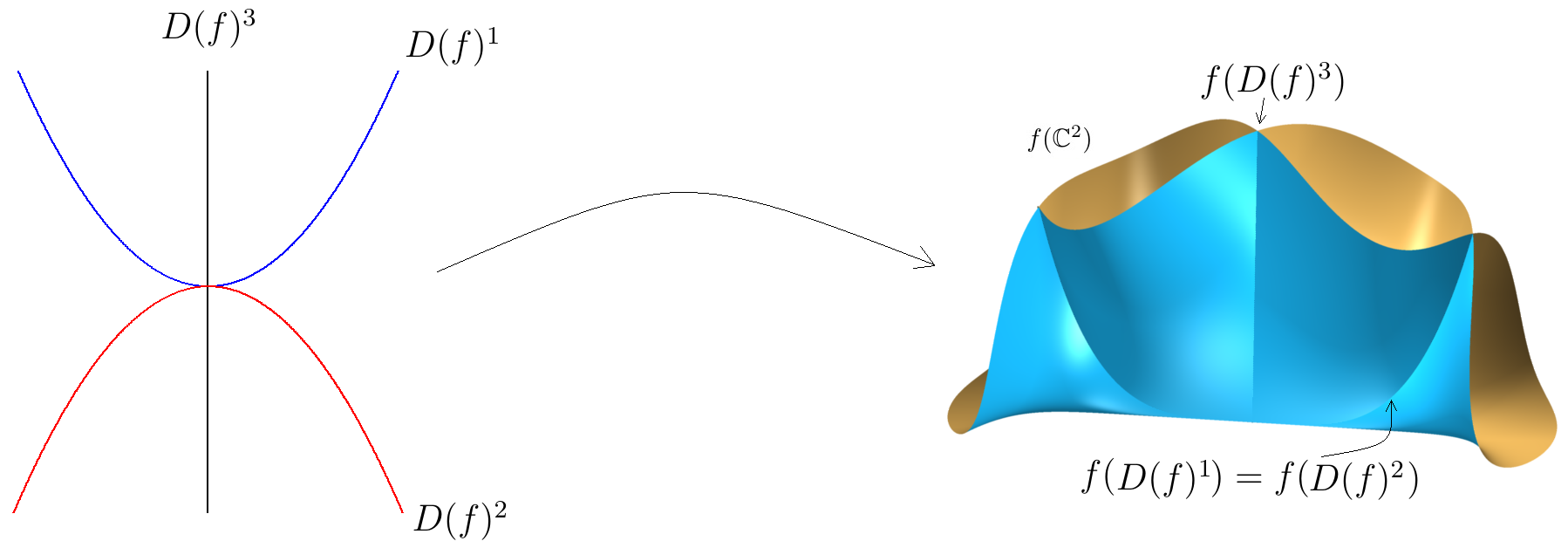}  
\caption{Identification and fold components of $D(f)$ (real points)}\label{figura38}
\end{figure}

\end{example}

\begin{remark} To compute $D(f)$ and $f(D(f))$, one can use Mond-Pellikaan's algorithm \rm\cite[\textit{Section 2}]{ref13} \textit{to find presentation matrices. For the computations in Example} \rm\ref{example}\textit{, we have made use of the software Singular} \rm\cite{singular} \textit{and the implementation of Mond-Pellikaan's algorithm given by Hernandes, Miranda, and Pe\~{n}afort-Sanchis in} \rm\cite{ref6}. \textit{At the webpage of Miranda} \rm\cite{ref10} \textit{one can find a Singular library to compute presentation matrices based on the results of} \rm\cite{ref6}. \textit{We also have made use of the software Surfer }\rm\cite{surfer}\textit{.} 
\end{remark}

We recall that a polynomial $p(x_1,\cdots,x_n)$ is \textit{quasihomogeneous} if there are positive integers $w_1,\cdots,w_n$, with no common factor and an integer $d$ such that $p(k^{w_1}x_1,\cdots,k^{w_n}x_x)=k^dp(x_1,\cdots,x_n)$. The number $w_i$ is called the weight of the variable $x_i$ and $d$ is called the weighted degree of $p$. In this case, we say $p$ is of type $(d; w_1,\cdots,w_n)$.

This definition extends to polynomial map germs $f:(\mathbb{C}^n,0)\rightarrow (\mathbb{C}^p,0)$ by just requiring each coordinate function $f_i$ to be quasihomogeneous of type $(d_i; w_1,\cdots,w_n)$. In particular, we say that $f:(\mathbb{C}^2,0)\rightarrow (\mathbb{C}^3,0)$ is quasihomogeneous of type $(d_1,d_2,d_3; w_1,w_2)$.

\begin{remark} Choose an integer $b$. Note that $g(x,y)=(x,y^2,xy)$ can be considered as a quasihomogeneous map germ of type $(1,2b,b+1; 1,b)$. To avoid any confusion, in this work we consider $g$ of type $(1,2,2; 1,1)$, that is, $b=1$.
\end{remark}

Given a finitely determined map germ $f:(\mathbb{C}^2,0)\rightarrow (\mathbb{C}^3,0)$, a natural question is: \textit{how many identification and fold components does the curve $D(f)$ have?}\\

In the case where $f(x,y)=(x,p(x,y),q(x,y))$ is homogeneous, Ruas and the author presented in \cite[Proposition 6.2]{otoniel1} an answer to this question in terms of the degrees of $p$ and $q$. In other words, \cite[Proposition 6.2]{otoniel1} give us an answer for the class of quasihomogeneous maps of type $(1,d_2,d_3; 1, 1)$. In  Proposition \ref{propdasformulas}, we extend this result for quasihomogeneous maps of type $(1,d_2,d_3; 1, w_2)$, with $w_2\geq 1$.

\begin{lemma}\label{lemapropformulas} Let $f:(\mathbb{C}^2,0)\rightarrow (\mathbb{C}^3,0)$ be a finitely determined map germ of corank $1$. Write $f$ in the form $f(x,y)=(x,p(x,y),q(x,y))$ and suppose that $p$ and $q$ are quasihomogeneous functions of types $(d_2; 1,b)$ and $(d_3; 1, b)$, respectively,  with $b \geq 1$. Let $f:U\rightarrow W$ a small representative. Then\\

\noindent (a) $D(f)=V(\lambda(x,y))$, where $\lambda$ is a quasihomogeneous polynomial of type $(d; 1, b)$ where $d= \dfrac{d_2d_3}{b}-d_2-d_3+b$. Furthermore, $b$ divides $d_2$ or $d_3$.\\ 

\noindent (b) The polynomial $\lambda$ can take the following forms:

 \begin{center}
(I) $\lambda(x,y)=\displaystyle { \prod_{i=1}^{d/b}}(y-\alpha_i x^b)$, $ \ \ \ \ \  \ $ (II) $\lambda(x,y)= x \left( \displaystyle {   \prod_{i=1}^{\frac{d-1}{b}}}(y-\alpha_i x^b) \right) $ $ \ \ \ \ $ or $ \ \ \  \ $ (III) $\lambda(x,y)=x$,
\end{center}

\noindent where $\alpha_{1},\cdots, \alpha_{d/b}$ (respectively, $\alpha_{1},\cdots, \alpha_{(d-1)/b}$) are all distinct in (I) (respectively in (II)).\\ 

\noindent (c) Let $\mathscr{C}_{\alpha}:=V(y-\alpha x^b)$ and $\mathscr{C}:=V(x)$ be plane curves, with $\alpha \in \mathbb{C}$. If $\mathscr{C}_{\alpha} \subset D(f)$ (respectively $\mathscr{C} \subset D(f)$), then $\mathscr{C}_{\alpha}$ is an identification component of $D(f)$ (respectively, $\mathscr{C}$ is a fold component of $D(f)$). 

\end{lemma}

\begin{proof} The first part of (a) is proved in \cite[Prop. 1.15]{mond3}. We have that $V(F_2(f_{\ast}\mathcal{O}_2))$ is zero dimensional by Remark \ref{remarktriplepoints}, thus it follows by \cite[Lemma 1.13]{mond3} that $b$ divides $d_2$ or $d_3$.\\

\noindent (b) By (a) we have that $D(f)=V(\lambda(x,y))$, where $\lambda(x,y)$ is a quasihomogeneous polynomial of type $(1,b;d)$. 
It is not hard to see that the only irreducible quasihomogeneous polynomials in $\mathbb{C}[x,y]$ with weights $w(x)=1$ and $w(y)=b$ are $x$, $y$ and $y-\gamma x^b$, with $0\neq \gamma \in \mathbb{C}$. Since the ring of polynomials $\mathbb{C}[x,y]$ is an unique factorization domain, each irreducible factor of $\lambda$ is on the form $x$ or $y-\gamma x^b$ (with $\gamma$ not necessarily non-zero). Since $f$ is finitely determined, by Theorem \ref{criterio} $\lambda$ is reduced, that is, the factors of $\lambda$ are all distinct. So the function $\lambda$ can take the following forms:  

 \begin{center}
(I) $\lambda(x,y)=\displaystyle { \prod_{i=1}^{r_1}}(y-\alpha_i x^b)$, $ \ \ \ \ \  \ $ (II) $\lambda(x,y)= x \left( \displaystyle {   \prod_{i=1}^{r_2}}(y-\alpha_i x^b) \right) $ $ \ \ \ \ $ or $ \ \ \  \ $ (III) $\lambda(x,y)=x$,
\end{center}

\noindent where $\alpha_{1},\cdots, \alpha_{r_1}$ (respectively, $\alpha_{1},\cdots, \alpha_{r_2}$) are all distinct in (I) (respectively in (II)). Now, since $\lambda$ is quasihomogeneous of type $(1,b;d)$, thus $r_1=d/b$ and $r_2=(d-1)/b$.\\

\noindent (c) Consider the parametrizations $\varphi_{\alpha},\varphi: W \rightarrow U \subset \mathbb{C}^2$ of $\mathscr{C}_{\alpha}$ and $\mathscr{C}$, defined by $\varphi_{\alpha}(u)=(u,\alpha u^b)$ and $\varphi(u):=(0,u)$, respectively, where $W$ is a neighborhood of $0 \in \mathbb{C}$. Note that $f \circ \varphi_{\alpha} =(u,p(u, \alpha u^b),q(u, \alpha u^b))$ is $1$-to-$1$. Thus if $\mathscr{C}_{\alpha} \subset D(f) $ then $\mathscr{C}_{\alpha}$ is an identification component of $D(f)$. On the other hand, note that $f \circ \varphi =(0,p(0, \alpha u^b),q(0, \alpha u^b))$, thus $f(\mathscr{C}_{\alpha}) \neq f(\mathscr{C})$ for all $\alpha \in \mathbb{C}$. Hence, if $\mathscr{C} \subset D(f)$ then $\mathscr{C}$ necessarily is a fold component of $D(f)$.\end{proof}\\

Now, we present the main result of this section.

\begin{proposition}\label{propdasformulas} Let $f,b,d_2,d_3$ and $\lambda$ be as in Lemma \ref{lemapropformulas}.\\

\noindent (a) $\lambda$ can be written in the form (III) (ie; $\lambda(x,y)=x$) if and only if $f$ is $\mathcal{A}$-equivalent to $g(x,y)=(x,y^2,xy)$. In this case $V(x)$ is a fold component of $D(f)$, which is irreducible.\\

\noindent (b) Suppose that $b$ is odd.\\

(b.1) $\lambda$ can be written in the form (I) if and only if $d_2$ or $d_3$ is odd. In this case, $D(f)$ has $d/b$ identification components and there are no fold components.

(b.2) $\lambda$ can be written in the form (II) if and only if $d_2$ and $d_3$ are both even and $d \neq 1$. In this case, $D(f)$ has $(d-1)/b$ identification components and $V(x)$ is the only fold component of $D(f)$.\\

\noindent (c) Suppose that $b$ is even. In this case, either $d_2$ is even or $d_3$ is even.\\

(c.1) $\lambda$ can be written in the form (I) if and only if $d_2$ and $d_3$ are both even. In this case, $D(f)$ has $d/b$ identification components and there are no fold components.

(c.2) $\lambda$ can be written in the form (II) if and only if $d_2$ and $d_3$ have different parity and $d\neq 1$. In this case, $D(f)$ has $(d-1)/b$ identification components and $V(x)$ is the only fold component of $D(f)$.

\end{proposition}

\begin{proof} (a) Consider the map germ $g(x,y)=(x,y^2,xy)$. Note that $D(g)=V(x)$. On the other hand, suppose that $D(f)=V(x)$, thus $\mu(D(f))=0$. By \cite[Theorem 3.4]{ref7} (where $D(f)$ and $D^2(f)$ are denoted by $D^2(f)$ and $\tilde{D}^2(f)$, respectively), we have that $\mu(D^2(f))=0$, so $D^2(f)$ is a smooth curve. Hence $f$ is stable by \cite[Proposition 2.14]{ref7}. It is well known that the only stable map germ from $\mathbb{C}^2$ to $\mathbb{C}^3$ is $g$ (up $\mathcal{A}$-equivalence), it follows that $f\sim_{\mathcal{A}}g$.\\

We recall also that the number of identification components of $D(f)$ is always even, thus by Lemma \ref{lemapropformulas} (c) we have that:\\

\noindent\textbf{Fact 1:} $\lambda$ can be written in the form (I) if and only if $d$ is even.

\noindent\textbf{Fact 2:} $\lambda$ can be written in the form (II) if and only if $d$ is odd and $d \neq 1$.\\

In order to prove (b) and (c), we have also the obvious fact:\\

\noindent\textbf{Fact 3:} $\lambda$ can be written in the form (III) if and only if $d=1$.\\

\noindent (b) By Prop. \ref{propdasformulas} (a), we have that $b$ divides $d_2$ or $d_3$. Without loss of generality we can assume that $b$ divides $d_2$ and set $d_2=nb$. Thus it easy to see that $d=nd_3-nb-d_3+b$ is odd if and only if $n$ and $d_3$ are both even. Since $b$ is odd, $d_2$ is even if and only if $n$ is even. Thus, $d$ is odd if and only if $d_2$ and $d_3$ are both even. Now, the statements (b.1) and (b.2) follows by Facts (I) and (II).\\

\noindent (c) By Prop. \ref{propdasformulas} (a), we have that $b$ divides $d_2$ or $d_3$. Suppose first that $b$ divides $d_2$ and set $d_2=nb$, in particular, $d_2$ is even. Thus we have the following cases illustrated in Table \ref{tabela1}:

\begin{table}[!h]
\caption{Cases of parity}\label{tabela1}
\centering 
{\def\arraystretch{2}\tabcolsep=10pt
\begin{tabular}{@{}l | c  |  c | c | c @{}}
\hline
Case &  $d_2$ & $n$ & $d_3$ & $d=nd_3-nb-d_3+b$\\
\hline
Case (1) & even & even & even & even \\

Case (2) & even & even & odd & odd \\

Case (3) & even & odd & even & even \\

Case (4) & even & odd & odd & even \\

\end{tabular}
}
\end{table}

Thus $d$ is odd if and only if ($d_2$ is even), $n$ is even and $d_3$ is odd. We affirm the following:\\

\noindent \textbf{Statement:} There is no finitely determined map germ $f$ in the conditions of the case (4).\\ 

\noindent \textit{Proof of the Statement:} Since $d$ is even, by Lemma \ref{lemapropformulas} (c) we have that each irreducible component of $D(f)$ is an identification component. Write $f$ in the form

\begin{center}
$f(x,y)=(x \ , \ x\tilde{p}(x,y) + \alpha y^n \ , \ x\tilde{q}(x,y)+ \beta y^{m})$
\end{center}

\noindent that is, $p(x,y)=x \tilde{p}(x,y) + \alpha y^n$ and $q(x,y)=x \tilde{q}(x,y)+ \beta y^{m}$. Notice that $\alpha \neq 0$ or $\beta \neq 0$, otherwise $f$ is not finite. We can suppose that $2\leq n,m$, otherwise $corank(f)=0$. Consider the curve $\mathscr{C}=V(x)$ and the following parametrization $\varphi(u)=(0,u)$ of $\mathscr{C}$. Note that $(f \circ \varphi ) (u) =(0,\alpha y^n, \beta y^m)$. Since $D(f)$ has no fold components, $\mathscr{C}$ is not an irreducible component of $D(f)$. Since $f$ is $1$-to-$1$ outside $D(f)$, $f\circ \varphi$ is $1$-to-$1$. This implies that $\beta \neq 0$ (and $gcd(n,m)=1$). Since $\beta \neq 0$, we have that $d_3=mb$ which is even, a contradiction.\\ 

Thus, since the case (4) does not occurs, we have that:

\begin{flushleft}
(1) $d$ is odd if and only if ($d_2$ is even) and $d_3$ is odd.\\

(2) $d$ is even if and only if ($d_2$ is even) and $d_3$ is even. 
\end{flushleft}

In a similar way, if we suppose that $b$ divides $d_3$, we conclude that 

\begin{flushleft}
(1) $d$ is odd if and only if ($d_3$ is even) and $d_2$ is odd.\\

(2) $d$ is even if and only if ($d_3$ is even) and $d_2$ is even. 
\end{flushleft}

Now, the statements (c.1) and (c.2) follows by Facts (I), (II) and (III). Finally, we remark that the statement on the number of identification and fold components in (b.1), (b.2), (c.1) and (c.2) follows by Lemma \ref{lemapropformulas} (c).\end{proof}

\begin{example}\label{example2} Consider $f(x,y)=(x,y^2,xy^3-x^5y)$ the map germ of Example \rm\ref{example}. \textit{Note that $f$ is quasihomogeneous of type $(1,4,7; 1,2)$. By Proposition \ref{propdasformulas} (c.2) we conclude that $D(f)$ has two identification components and $V(x)$ is the only fold component, exactly according to what we present in Example} \rm\ref{example}.
\end{example}

\section{Equisingularity in families of map germs}\label{sec:3}

$ \ \ \ \ $ Let $f:(\mathbb{C}^2,0)\rightarrow (\mathbb{C}^3,0)$ be a finitely determined map germ. Let $F:(\mathbb{C}^{2} \times \mathbb{C},0)\rightarrow (\mathbb{C}^3 \times \mathbb{C},0)$ be a $1$-parameter unfolding of $f$ defined by $F(x,t)=(f_{t}(x),t)$. We assume that the origin is preserved, that is, $f_{t}(0)=0$ for all $t$. 

As in \cite[Section 4]{ref1}, in a similar way we can also define the double point space $D^2(F)$ of $F$, which in this work, is considered as a family of curves in $(\mathbb{C}^2 \times \mathbb{C}^2 \times \mathbb{C},0)$ (instead of $(\mathbb{C}^3 \times \mathbb{C}^3,0)$). We consider also the other double point spaces $D(F)$ in $(\mathbb{C}^2 \times \mathbb{C},0)$ and $F(D(F))$ in $(\mathbb{C}^3 \times \mathbb{C},0)$ (again both as families of curves).

Consider a small representative $F:U \times T \rightarrow V \times T$ of $F$, where $U \times T$, $V \times T$ and $T$ are neighborhoods of $0$ in $\mathbb{C}^2 \times \mathbb{C}$, $\mathbb{C}^3 \times \mathbb{C}$ and $\mathbb{C}$, respectively. Gaffney defined in \cite[Def. 6.2]{gaffney} the excellent unfoldings. An excellent unfolding has a natural stratification whose strata in the complement of the parameter space $T$ are the stable types in source and target. In our case, the strata in the source are the following:

\begin{equation}\label{eq111}
\lbrace U \times T \setminus D(F), \ D(F) \setminus T, \ T \rbrace.
\end{equation}

\noindent where $D(F)$ denotes the set of double point space of $F$ (the representative of the germ $F$). In the target, the strata are

\begin{equation}\label{eq222}
\lbrace V \times T \setminus F(U \times T), \  F(U \times T) \setminus \overline{F(D(F))}, \  F(D(F)) \setminus T, \ T \rbrace.\\
\end{equation}

Notice that $F$ preserves the stratification, that is, $F$ sends a stratum into a stratum.

\begin{definition}\label{defequisingularity} Let $F:(\mathbb{C}^2 \times \mathbb{C},0)\rightarrow (\mathbb{C}^3 \times \mathbb{C},0)$ be a $1$-parameter unfolding of a finitely determined map germ $f:(\mathbb{C}^2,0)\rightarrow(\mathbb{C}^3,0)$.\\ 

\noindent (a) We say that $F$ is \textit{topologically trivial} if there are germs of homeomorphisms:
\[
\Phi :(\mathbb{C}^2 \times \mathbb{C},0)\rightarrow (\mathbb{C}^2 \times \mathbb{C},0),  \ \Phi(x,t)=(\phi_{t}(x),t), \ \phi_{0}(x)=x, \ \phi_{t}(0)=0 
\]
\[
\Psi:(\mathbb{C}^3 \times \mathbb{C},0)\rightarrow (\mathbb{C}^3 \times \mathbb{C},0), \ \Psi(x,t)=(\psi_{t}(x),t), \ \psi_{0}(x)=x, \ \psi_{t}(0)=0
\]
\noindent such that $I=\Psi^{-1} \circ F \circ \Phi$, where $I(x,t)=(f(x),t)$ is the trivial unfolding of $f$.\\

\noindent (b) We say that $F$ is \textit{Whitney equisingular} if there is a representative of $F$ such that the stratifications in (\ref{eq111}) and (\ref{eq222}) are Whitney regular along $T$.
\end{definition}

\begin{remark}\label{remarkprincipal}(a) We have that the projections of $\pi_1:D(F)\rightarrow T$ and $\pi_2:F(D(F))\rightarrow T$ into the parameter space $T$ are flat deformations of $D(f)$ and $f(D(f))$ (see \rm\cite[\textit{Lemma 4.2}]{ref9}\textit{), that is, $D(F)$ and $F(D(F))$ are flat families of reduced curves in the sense of} \rm\cite[\textit{Section 5}]{buch}.\\

\noindent \textit{(b) Suppose that $F$ is topologically trivial. Since the representatives of the germs of homeomorphisms $\Phi$ and $\Psi$ (in Def. \ref{defequisingularity}(a)) must preserve the double point curves, then $D(F)$ and $F(D(F))$ are topologically trivial families of curves, that is, there are homeomorphisms} 

\begin{center}
 $\Phi: D(F) \rightarrow D(f) \times T$ $ \ \ \ $ \textit{and} $ \ \ \ $ $\Psi: F(D(F)) \rightarrow f(D(f)) \times T$,
 \end{center} 

\noindent \textit{with $\pi_1^{'} \circ \Phi=\pi_1$ and $\pi_2^{'} \circ \Psi=\pi_2$, where $p_1^{'}$ and $p_2^{'}$ are the canonical projections on the second factor of $D(f) \times T$ and $f(D(f)) \times T$, respectively.}\\

\noindent \textit{(c) By Thom's second isotopy lemma for complex analytic maps (} \rm\cite[\textit{Th. 5.2}]{ref5}\textit{), every unfolding $F$ of $f$ which is Whitney equisingular is also topologically trivial. However, we know that the converse is not true in general (see} \cite[Section 5]{otoniel1}\textit{).}

\end{remark}

The following theorem characterizes Whitney equisingularity (in corank $1$ case) and topological triviality. 

\begin{theorem}\label{18}
\textit{Let $F:(\mathbb{C}^2 \times \mathbb{C},0)\rightarrow (\mathbb{C}^3 \times \mathbb{C},0)$ be a $1$-parameter unfolding of a finitely determined map germ $f:(\mathbb{C}^2,0)\rightarrow(\mathbb{C}^3,0)$. Then:}

\begin{flushleft}
\textit{(a) $F$ is topologically trivial if and only if $\mu(D(f_t))$ is constant.}\\

\textit{(b) If $f$ has corank $1$, then $F$ is Whitney equisingular if and only if $\mu(D(f_t))$ and $m(f_t(D(f_t)))$ are constant.}

\end{flushleft}

\end{theorem}

Theorem \ref{18} (a) is proved in \cite[Corollary 32]{ref2} (see also \cite[Theorem 6.2]{ref1}). Theorem \ref{18} (b) is proved in \cite[Corollary 8.9]{gaffney}, although it is stated there in terms of the invariant $e_D(f_t)=\mu(D(f_t))+2m(f_t(D(f_t)))-1$. In \cite[Th. 8.7]{gaffney} and \cite[Th. 5.3]{ref9} the reader can find characterizations of Whitney equisingularity which also include the corank $2$ case.\\

Now, we present the main result of this section which is an application of Proposition \ref{propdasformulas} and Lemma \ref{lemapropformulas}.

\begin{theorem}\label{mainresult1} Let $f:(\mathbb{C}^2,0)\rightarrow (\mathbb{C}^3,0)$ be a finitely determined map germ of corank $1$. Write $f$ in the form $f(x,y)=(x,p(x,y),q(x,y))$ and suppose that $p$ and $q$ are quasihomogeneous functions such that the weights of the variables are $w(x)=1$ and $w(y)=b \geq 2$. Let $F=(f_t,t)$ be a $1$-parameter unfolding of $f$. If $F$ is topologically trivial, then $F$ is Whitney equisingular. In particular, the family $F(\mathbb{C}^2 \times \mathbb{C})$ is equimultiple.
\end{theorem}

\begin{proof} Note that if $\varphi(u)=(u^m,\varphi_2(u),\varphi_3(u))$ is a Puiseux parametrization (see \cite[p. 98]{chirka}) of a reduced curve in $\mathbb{C}^3$, then its multiplicity is $m$. We have that $D(f)$ is a reduced curve, by Theorem \ref{criterio}. Thus $f(D(f))$ is also a reduced curve, by \cite[Th. 4.3]{ref9}. By \cite[Lemma 7.1]{otoniel1}, it is sufficient to prove that the image of each irreducible component of $f(D(f))$ has multiplicity $1$. Hence, to calculate the multiplicity of the image of each irreducible component of $D(f)$, it is enough to find its parametrizations. 

By Lemma \ref{lemapropformulas}, we have that the irreducible components of $D(f)$ are given by either $V(x)$ or $V(y-\alpha x^b)$, with $\alpha \in \mathbb{C}$. As in Lemma \ref{lemapropformulas} (c), consider the curves $\mathscr{C}_{\alpha}:=V(y-\alpha x^b)$ and $\mathscr{C}:=V(x)$. Consider the Puiseux parametrizations $\varphi_{\alpha}(u):=(u,\alpha u^b)$ and $\varphi(u):=(0,u)$ of $\mathscr{C}_{\alpha}$ and $\mathscr{C}$, respectively. Suppose that $\mathscr{C}_{\alpha} \subset D(f)$. Note that $f \circ \varphi_{\alpha} (u)=(u,p(u,\alpha u^b),q(u,\alpha u^b))$ is a Puiseux parametrization of $f(\mathscr{C}_{\alpha})$. Thus, the multiplicity of $f(\mathscr{C}_{\alpha})$ is $1$ and thus $f(\mathscr{C}_{\alpha})$ is smooth.

Suppose now that $\mathscr{C} \subset D(f)$. We have to show that $f(\mathscr{C})$ is a smooth curve. By Lemma \ref{lemapropformulas} (c), we have that $\mathscr{C}$ is a fold component of $D(f)$. Write $f$ in the form

\begin{center}
$f(x,y)=(x \ , \ x\tilde{p}(x,y) + \alpha y^n \ , \ x\tilde{q}(x,y)+ \beta y^{m})$
\end{center}

\noindent that is, $p(x,y)=x\tilde{p}(x,y) + \alpha y^n$ and $q(x,y)=x\tilde{q}(x,y)+ \beta y^{m}$. Note that $\alpha \neq 0$ or $\beta \neq 0$ and $2 \leq n,m$. Hence, we have three cases:

\begin{flushleft}
\textit{Case (1)} $f \circ \varphi(u)=(0, \alpha u^n,0)$, if $\alpha \neq 0$ and $\beta=0$.\\
\textit{Case (2)} $f \circ \varphi(u)=(0,0,\beta u^m)$, if $\alpha=0$ and $\beta \neq 0$.\\
\textit{Case (3)} $f \circ \varphi(u)=(0, \alpha u^n,\beta u^m)$, if $\alpha \neq 0$ and $\beta \neq 0$.
\end{flushleft}

Let's look at each case.\\

\noindent \textit{Case (1)}: We have that $f \circ \varphi(u)=(0,\alpha u^n,0)$ is generically $n-$to$-1$. Since $\mathscr{C}$ is a fold component of $D(f)$, thus $n=2$.\\

\noindent \textit{Case (2)}: This case is analogous to the previous one.\\

Notice that in the case (1) (respectively, case (2)) the map $\tilde{\varphi}:W\rightarrow V$ defined by $\tilde{\varphi}(u):=(0,\alpha u, 0)$ (respectively, $\tilde{\varphi}(u):=(0,0, \beta u)$) is a Puiseux parametrization of $f(\mathscr{C})$. So, we have that $f(\mathscr{C})$ has multiplicity $1$ in the cases (1) and (2), thus it is smooth in both cases.\\

\noindent \textit{Case (3)}: Note that $d_2=bn$ and $d_3=bm$, where $d_2,d_3$ are the weighted degrees of $p$ and $q$. Thus, $D(f)=V(\lambda(x,y))$, where $\lambda$ has weighted degree equal to $d=bnm-bn-bm+b$. Let's split this case into two parts.\\

Part 1: Suppose that $b$ is odd. We affirm that:\\

\textbf{Statement:} If $b\geq 3$, then $f$ is not finitely determined.\\

\noindent \textit{Proof of the Statement}: Note that $\lambda$ can be written in the form (II) of Lemma \ref{lemapropformulas}. Since the weighted degree of $x$ is $1$, it follows that $d$ is congruent to $1$ modulo $b$. The hypothesis that $b\geq 3$ implies that $0$ and $1$ are not congruent modulo $b$ (this happens just in the case where $b=1$). We have that $d=b(nm-n-m+1)$, thus $d$ is congruent to $0$ modulo $b$, a contradiction.\\

Part 2: Suppose that $b$ is even. By Prop. \ref{propdasformulas} (c), we have that $d_2$ and $d_3$ have different parity. Suppose that $\alpha,\beta \neq 0$. Thus $b_2=bn$ and $b_3=bm$ which are both even. Hence we have that either $\alpha=0$ or $\beta = 0$. Now, we proceed as in case (1) and conclude that $f(\mathscr{C})$ is smooth.\end{proof}

\section{Equimultiplicity of families of map germs}\label{sec5}

$ \ \ \ \ $ We start with two technical lemmas which will be useful to prove the main result of this section.

\begin{lemma}\label{lemaaux3} Let $f:(\mathbb{C}^2,0)\rightarrow (\mathbb{C}^3,0)$ be a finitely determined map germ. Let $F(x,y,t)=(f_t(x,y),t)$ be a $1$-parameter unfolding of $f$ which is topologically trivial. Consider a small representative $F:U \times T \rightarrow V \times T$ of $F$, where $U \times T$, $V \times T$ and $T$ are neighborhood of $0$ in $\mathbb{C}^2 \times \mathbb{C}$, $\mathbb{C}^3 \times \mathbb{C}$ and $\mathbb{C}$, respectively. Then $F^{-1}(\lbrace (0,0,0) \rbrace \times T)=\lbrace (0,0) \times T \rbrace$. 
\end{lemma}

\begin{proof} The proof is consequence of part of the proof of \cite[Theorem 8.7]{gaffney}. So, we proceed as in the proof \cite[Th. 8.7]{gaffney} and include the proof for completeness.\\

Since $F$ is topologically trivial, by Theorem \ref{18} we have that $\mu(D(f_t))$ is constant. By \cite[Theorem 4.3]{ref9} and the fact that the invariants $\mu(D(f)),\mu(f(D(f)), C(f)$ and $T(f)$ are upper semicontinuos we also conclude that $C(f_t)$ and $T(f_t)$ are constant.

Suppose that $F^{-1}( \lbrace (0,0,0) \rbrace \times T) \neq \lbrace (0,0) \times T \rbrace$ on any neighborhood of $(0,0)$ in $\mathbb{C}^2 \times \mathbb{C}$. If the points of  $F^{-1}( \lbrace (0,0,0) \rbrace \times T)$ lie in the ramification set $\Sigma(F)$, then $C(f_t)$ must change at the origin so we can assume that they lie in  $F^{-1}( \lbrace (0,0,0) \rbrace \times T) \setminus (\lbrace (0,0) \times T \rbrace \cup \Sigma(F))$. 

Let $\overline{x}=(x_0,y_0,t_0)$ be a point in $F^{-1}( \lbrace (0,0,0) \rbrace \times T) \setminus (\lbrace (0,0) \times T \rbrace \cup \Sigma(F))$ and let $U_0$ and $U_{\overline{x}}$ be neighborhoods of $0=(0,0,t_0)$ and $\overline{x}$, respectively, such that $U_0, U_{\overline{x}} \subset U \times \lbrace t_0 \rbrace$. Consider the images $f_t(U_0)$ and $f_t(U_{\overline{x}})$. By hypothesis, we have that $f(0)=f(\overline{x})$. Note that the intersection $f_t(U_0)\cap f_t(U_{\overline{x}})$ cannot be $2$-dimensional, otherwise $D(f_t)$ is $2$-dimensional and $f_t$ is not finitely determined. Hence $f$ would not be finitely determined if this held for all $t$ close to zero. So, $f_t(U_0)\cap f_t(U_{\overline{x}})$ is a curve. If $f_t(U_0)\cap f_t(U_{\overline{x}})$ lies in the singular set of $f_t(U_0)$ then the set of triple points is at least one dimensional, hence $f$ would not be finitely determined if this held for all $t$ close to zero. Finally, if $f_t(U_{\overline{x}})$ meets the singular set of $f_t(U_0)$ properly, then $\overline{x}$ should be a singular point $D(f_t)$, hence $\mu(D(f_t))$ must jump at $0$.\end{proof}\\

The following result shows us how we can relate the equimultiplicity of a family $f_t$ in terms of the Cohen-Macaulay property of a certain local ring. In the sequel, as in \cite[p. 12, 107, 20]{matsumura}, $l(R)$, $e(q,A)$ and $A_p$ denotes the length of the $0$-dimensional local ring $R$, the Hilbert-Samuel multiplicity of the ideal $q$ in the local ring $A$ and the localization of $A$ in the prime ideal $p$, respectively. As before, $m(f_t(\mathbb{C}^2))$ denotes the multiplicity of $f_t(\mathbb{C}^2)$ at $0$.

\begin{lemma}\label{lemaaux} Let $f:(\mathbb{C}^2,0)\rightarrow (\mathbb{C}^3,0)$ be a finitely determined map germ with corank $1$. Let $F(x,y,t)=(f_t(x,y),t)$ be a $1$-parameter unfolding of $f$. Then

\begin{flushleft}
(a) the ideal $ \langle t \rangle $ in $\mathbb{C} \lbrace x,y,t \rbrace / \langle F^{\ast}(X), F^{\ast}(Y),F^{\ast}(Z) \rangle $ is a parameter ideal and
\end{flushleft}

\begin{center}
$ m(f(\mathbb{C}^2)) = \displaystyle {l \left( \dfrac{\mathbb{C} \lbrace x,y,t \rbrace}{ \langle F^{\ast}(X), F^{\ast}(Y),F^{\ast}(Z),t \rangle} \right)}$ $\hspace{0.5cm} $ and $ \hspace{0.5cm} m(f_t(\mathbb{C}^2)) = \displaystyle {e \left( \langle t \rangle , \dfrac{\mathbb{C} \lbrace x,y,t \rbrace}{ \langle F^{\ast}(X), F^{\ast}(Y),F^{\ast}(Z) \rangle } \right)}$ for $t \neq 0$.
\end{center}

\noindent where $F^{\ast}:\mathbb{C}\lbrace X,Y,Z,t \rbrace \rightarrow \mathbb{C}\lbrace x,y,t \rbrace$ denotes the morphism of local rings associated to $F$.

\begin{flushleft}
(b) $m(f_t(\mathbb{C}^2))$ is constant if and only if $\mathbb{C} \lbrace x,y,t \rbrace / \langle F^{\ast}(X), F^{\ast}(Y),F^{\ast}(Z) \rangle $ is a Cohen-Macaulay ring.
\end{flushleft}

Where $(x,y,t)$ are the coordinates of $\mathbb{C}^2 \times \mathbb{C}$ and $(X,Y,Z,t)$ are the coordinates of $\mathbb{C}^3 \times \mathbb{C}$.

\end{lemma}

\begin{proof} Consider a small representative $F:U \times T \rightarrow V \times T$ of $F$, where $U \times T$, $V \times T$ and $T$ are neighborhood of $0$ in $\mathbb{C}^2 \times \mathbb{C}$, $\mathbb{C}^3 \times \mathbb{C}$ and $\mathbb{C}$, respectively.\\

(a) Since $F$ is topologically trivial, by Lemma \ref{lemaaux3} we have that $V(\langle F^{\ast}(X), F^{\ast}(Y),F^{\ast}(Z) \rangle)\cap U \times T = (0,0) \times  T$ in the source. Thus, we have the following equality between ideals

\begin{center}
$ \sqrt{ \langle F^{\ast}(X), F^{\ast}(Y),F^{\ast}(Z) \rangle } \mathbb{C}\lbrace x,y \rbrace = \langle x,y \rangle \mathbb{C}\lbrace x,y \rbrace $,
\end{center}

\noindent and therefore $x^n,y^m \in \langle F^{\ast}(X), F^{\ast}(Y),F^{\ast}(Z) \rangle$ for some $n,m \in \mathbb{N}$. Therefore,

\begin{center}
$ \sqrt{ \langle t  \rangle }  \left( \dfrac{\mathbb{C}\lbrace x,y,t \rbrace}{ \langle F^{\ast}(X), F^{\ast}(Y),F^{\ast}(Z) \rangle } \right)= \langle x,y,t \rangle \left( \dfrac{\mathbb{C}\lbrace x,y,t \rbrace}{ \langle F^{\ast}(X), F^{\ast}(Y),F^{\ast}(Z) \rangle } \right)$
\end{center}

\noindent and therefore $t$ is a parameter in $\mathbb{C} \lbrace x,y,t \rbrace / \langle F^{\ast}(X), F^{\ast}(Y),F^{\ast}(Z) \rangle $ (in the sense of \cite[Ch. $14$]{matsumura}). Since $f$ is finitely determined, $f$ is the normalization of its image $f(\mathbb{C}^2)$, thus (see for instance \cite[p. 210]{gaffney}) we have that:

\begin{equation}\label{eq1}
m(f(\mathbb{C}^2))=e(\langle f^{\ast}(X), f^{\ast}(Y),f^{\ast}(Z) \rangle ,\mathbb{C}\lbrace x,y \rbrace).
\end{equation}

Since $f$ has corank $1$, thus $ \langle f^{\ast}(X), f^{\ast}(Y),f^{\ast}(Z) \rangle \mathbb{C}\lbrace x,y \rbrace = \langle x,y^k \rangle \mathbb{C}\lbrace x,y \rbrace $ for some $k$, therefore

\begin{equation}\label{eq2}
e( \langle f^{\ast}(X), f^{\ast}(Y),f^{\ast}(Z) \rangle ,\mathbb{C}\lbrace x,y \rbrace)=e( \langle x,y^k \rangle ,\mathbb{C}\lbrace x,y \rbrace)= \displaystyle l \left( \dfrac{\mathbb{C}\lbrace x,y \rbrace}{ \langle x,y^k \rangle } \right)=l \left(\dfrac{\mathbb{C} \lbrace x,y,t \rbrace}{ \langle F^{\ast}(X), F^{\ast}(Y),F^{\ast}(Z),t \rangle } \right),
\end{equation}

\noindent it follows by (\ref{eq1}) and (\ref{eq2}) the first equality of (a). Since $f$ has corank $1$, we can write $f(x,y)=(x,p(x,y),q(x,y))$. Now, write 

\begin{equation}\label{eq4}
f_t(x,y)=(x+h_1(x,y,t),p(x,y)+h_2(x,y,t),q(x,y)+h_3(x,y,t))
\end{equation}
 
Since $F$ is has also corank $1$, after a change of coordinates we can assume that $h_1=0$ in (\ref{eq4}). Since $f_t$ is birrational over its image, we have that

\begin{center}
$m(f_t(\mathbb{C}^2))=e(\langle f_t^{\ast}(X), f_t^{\ast}(Y),f_t^{\ast}(Z)\rangle,\mathbb{C} \lbrace x,y \rbrace)=e(\langle x,y^s \rangle,\mathbb{C} \lbrace x,y \rbrace)=$
\end{center}

\begin{center}
$= \displaystyle l  \left(\dfrac{\mathbb{C}\lbrace x,y \rbrace}{\langle x,y^s \rangle} \right)=l \left(\dfrac{\mathbb{C}\lbrace x,y,t \rbrace_{ \langle x,y  \rangle}}{\langle F^{\ast}(X), F^{\ast}(Y),F^{\ast}(Z) \rangle } \right)=\displaystyle {e \left( \langle t \rangle , \dfrac{\mathbb{C} \lbrace x,y,t \rbrace}{\langle F^{\ast}(X), F^{\ast}(Y),F^{\ast}(Z) \rangle} \right)}$
\end{center}

\noindent for some $s\geq 1$, where the last equality follows by \cite[Th. 14.7]{matsumura}.\\

(b) Since the ideal $\langle t \rangle $ is a parameter ideal, by (a) and \cite[Th. 17.11]{matsumura} we have that $\mathbb{C} \lbrace x,y,t \rbrace / \langle F^{\ast}(X,Y,Z) \rangle $ is Cohen-Macaulay if and only if

\begin{center}
$ \displaystyle {l \left( \dfrac{\mathbb{C} \lbrace x,y,t \rbrace}{ \langle F^{\ast}(X), F^{\ast}(Y),F^{\ast}(Z),t \rangle } \right)}$ $ =  \displaystyle {e \left( \langle t \rangle , \dfrac{\mathbb{C} \lbrace x,y,t \rbrace}{ \langle F^{\ast}(X), F^{\ast}(Y),F^{\ast}(Z) \rangle } \right)}$,
\end{center}

\noindent if and only if $m(f(\mathbb{C}^2))=m(f_t(\mathbb{C}^2))$, for all $t$.\end{proof}\\

As an application of Lemma \ref{lemaaux}, we present the following theorem.

\begin{theorem}\label{mainresult2} Let $f:(\mathbb{C}^2,0)\rightarrow (\mathbb{C}^3,0)$ be a quasihomogeneous and finitely determined map germ of corank $1$. Let $F(x,y,t)=(f_t(x,y),t)$ be a topologically trivial $1$-parameter unfolding of $f$. If $F$ is an unfolding of non-decreasing weights, then the family $F(\mathbb{C}^2 \times \mathbb{C})$ is equimultiple.
\end{theorem}

\begin{proof} Write $f$ in the form

\begin{center}
$f(x,y)=(x \ , \ x\tilde{p}(x,y) + \alpha y^n \ , \ x\tilde{q}(x,y)+ \beta y^{m})$
\end{center}

\noindent Note that $\alpha \neq 0$ or $\beta \neq 0$, otherwise $f$ is not finite. As in the proof of Lemma \ref{lemaaux}, we can write $F$ in the following form:

\begin{center}
$F(x,y,t)=( \ x \ , \ \tilde{p}(x,y)+\alpha y^n  +h(x,y,t) \ , \ \tilde{q}(x,y)+\beta y^m +g(x,y,t) \ , \ t \ )$,
\end{center}

\noindent where $h(0,0,t)=g(0,0,t)=0$. Since $F$ is an unfolding of non-decreasing weights, we can write 

\begin{center}
$h(0,y,t)=h_0(t)y^r+h_1(t)y^{r+1}+\cdots$ $ \ \ \ $ and $ \ \ \ $ $g(0,y,t)=g_0(t)y^s+g_1(t)y^{s+1}+\cdots$,
\end{center} 

\noindent where $h_0(t),g_0(t)$ are not identically the zero function, $n\leqslant r$ and also $m\leqslant s$. So, we have that:

\begin{center}
$\dfrac{\mathbb{C}\lbrace x,y,t \rbrace }{\langle F^{\ast}(X), F^{\ast}(Y),F^{\ast}(Z) \rangle}\simeq \dfrac{\mathbb{C}\lbrace y,t \rbrace }{ \langle \alpha y^n+h_0(t)y^r+h_1(t)y^{r+1}+\cdots \ , \ \beta y^m +g_0(t)y^s+g_1(t)y^{s+1}+\cdots \rangle }$
\end{center}

We can suppose that $2 \leqslant n \leqslant m$. Hence, we have three cases:\\

\noindent \textit{ Case 1:} Suppose that $\alpha \neq 0$ and $\beta=0$. Since that $(\alpha+h_0(t)y^{r-n}+h_1(t)y^{r+1-n}+\cdots)$ is an invertible element in $\mathbb{C}\lbrace y,t \rbrace$ and $n\leqslant s$ we have that  

\begin{center}
$\dfrac{\mathbb{C}\lbrace x,y,t \rbrace }{\langle F^{\ast}(X), F^{\ast}(Y),F^{\ast}(Z) \rangle} \simeq \dfrac{\mathbb{C}\lbrace y,t \rbrace }{ \langle y^n \rangle } $,
\end{center}

\noindent which is a Cohen-Macaulay ring.\\

\noindent \textit{Case 2:} Suppose that $\alpha \neq 0$ and $\beta \neq 0$. As in the previous case, we conclude that $\mathbb{C}\lbrace x,y,t \rbrace / \langle F^{\ast}(X), F^{\ast}(Y),F^{\ast}(Z) \rangle $ is a Cohen-Macaulay ring.\\

\noindent \textit{Case 3:} Suppose that $\alpha=0$ and $\beta \neq 0$. Denote by $(\mathscr{C},0)$ the germ of curve defined by $(\mathscr{C},0)=V(x)$ in $U \subset \mathbb{C}^2$. Note that the restriction of $f$ to $\mathscr{C}$ is $m$-to-$1$, thus $m=1$ or $m=2$. Since we made the assumption that $m\geq 2$, the only option is $m=2$. If $r=1$, $f_t$ is an immersion, so $f_t(\mathbb{C}^2)$ is smooth and it has multiplicity $1$. We have that $f(\mathbb{C}^2)$ and $f_t(\mathbb{C}^2)$ are topologically equivalent, so $m(f(\mathbb{C}^2))=1$ (see for instance \cite[Th 3.4]{eyral2}). On the other hand, $f$ has corank $1$, and therefore $m(f(\mathbb{C}^2))>1$, a contradiction. Hence, we have that $r\geq 2$. Since $\beta+(g_0(t)+g_1(t)y+g_2(t)y^2+\cdots)$ is an invertible element in $\mathbb{C}\lbrace y,t \rbrace$ and $r\geq 2$, we have that:

\begin{center}
$\dfrac{\mathbb{C}\lbrace x,y,t \rbrace }{\langle F^{\ast}(X), F^{\ast}(Y),F^{\ast}(Z) \rangle} \simeq \dfrac{\mathbb{C}\lbrace y,t \rbrace }{ \langle y^2 \rangle } $,
\end{center}

\noindent which is also a Cohen-Macaulay ring.\\

\noindent Now, since $\mathbb{C}\lbrace x,y,t \rbrace / F^{\ast}(X,Y,Z)$ is a Cohen-Macaulay ring in all cases, $F$ is equimultiple, by Lemma \ref{lemaaux} (b).\end{proof}\\

We finish this work with some final remarks. 

\begin{remark} (a) We remark that by a result of Damon (\rm\cite[\textit{Corollary 1}]{damon}\textit{), every unfolding of non-decreasing weights is topologically trivial. An interesting question is: in the case of a quasihomogeneous finitely determined map germ $f:(\mathbb{C}^n,0)\rightarrow (\mathbb{C}^{n+1},0)$, an unfolding $F$ of $f$ is topologically trivial if and only if we can choose a system of coordinates such that $F$ is of non-decreasing weights?}\\

\noindent \textit{(b) The difficulty to extend Theorem} \rm\ref{mainresult2} \textit{without any hypothesis on the corank of $f$ is that if $f$ has corank $2$, then Lemma} \rm\ref{lemaaux} \textit{(a) is not more true. In fact, from} \rm\cite[Example 5.4]{ref9} \textit{consider the corank $2$ map germ $f(x,y)=(x^2,y^2,x^3+y^3+xy)$ and the unfolding $F=(f_t(x,y),t)$ of $f$ where $f_t$ is defined as}

\begin{center}
$f_t(x,y)=(x^2,y^2,x^3+y^3+xy+tx^3y+txy^3)$.
\end{center}

\noindent \textit{In} \rm\cite{ref9}, \textit{Marar, Nuño-Ballesteros and Peñafort-Sanchis showed that the unfolding $F$ is Whitney equisingular, in particular, $F$ is topologically trivial. Now, we have that}

\begin{center}
$ 4 = m(f(\mathbb{C}^2)) \neq \displaystyle {l \left( \dfrac{\mathbb{C} \lbrace x,y,t \rbrace}{ \langle x^2,y^2,x^3+y^3+xy+tx^3y+txy^3,t \rangle} \right)}=3$ $\hspace{0.5cm} $ and $ \hspace{0.5cm} 4=m(f_t(\mathbb{C}^2)) \neq \displaystyle {e \left( \langle t \rangle , \dfrac{\mathbb{C} \lbrace x,y,t \rbrace}{ \langle x^2,y^2,x^3+y^3+xy+tx^3y+txy^3 \rangle } \right)}=3$ for $t \neq 0$.
\end{center}

\noindent \textit{However, note that $\mathbb{C} \lbrace x,y,t \rbrace / \langle x^2,y^2,x^3+y^3+xy+tx^3y+txy^3 \rangle $ is a Cohen-Macaulay ring.}
\end{remark}

\begin{flushleft}
\textit{Acknowlegments:} Thanks are due to the referee for his very careful reading, valuable comments and suggestions. We would like to thank Professors M.A.S Ruas and J. Snoussi for many helpful conversations, suggestions and comments on this work. The author would like to thank CONACyT for the financial support by Fordecyt 265667 and UNAM/DGAPA for support by PAPIIT IN $113817$.
\end{flushleft}

\end{document}